\documentclass{au}
\usepackage{graphicx}
\usepackage{newlattice}

\newcommand{\Lh}{\widehat L}
\newcommand{\Ug}{\uparrow\!}
\newcommand{\SC}[1]{\msf{C}_{#1}}
\newcommand{\SM}[1]{\msf{M}_{#1}}

\theoremstyle{plain}
\newtheorem{theorem}{Theorem}
\newtheorem*{theoremn}{Theorem}
\newtheorem{lemma}[theorem]{Lemma}

\begin{document}
\title[On the structure theorem of planar semimodular lattices]{Notes on planar semimodular lattices. VI. \\On the structure theorem of planar semimodular lattices}  
\author{G. Gr\"{a}tzer} 
\email[G. Gr\"atzer]{gratzer@me.com}
\address{Department of Mathematics\\
  University of Manitoba\\
  Winnipeg, MB R3T 2N2\\
  Canada}
  \date{May 6, 2012}
\urladdr[G. Gr\"atzer]{http://server.maths.umanitoba.ca/homepages/gratzer/}


\begin{abstract}
In a recent paper, G. Cz\'edli and E.\,T. Schmidt present
a structure theorem for planar semimodular lattices. 
In this note, we present an alternative proof.
\end{abstract}

\maketitle

\section{Introduction}\label{S:I}
G. Cz\'edli and E.\,T. Schmidt \cite{CSb} present two structure theorems
for planar semimodular lattices (Theorem 3.6 and Corollary 3.5)
and a structure theorem for patch lattices, the building stones (Theorem 3.4). 
The first structure theorem for planar semimodular lattices 
is based on a new construction, patchwork systems, 
and the second on the classic gluing over a chain. 
Although our line of reasoning would simplify 
the proof of all three theorems, 
we discuss only the second structure theorem
for planar semimodular lattices because patchwork systems 
are technically complicated to introduce.

We start with a few concepts
from G.~Gr\"atzer and E. Knapp \cite{GK07} and \cite{GK09}.

Let $L$ be a planar lattice. 
A \emph{left weak corner} (resp., \emph{right weak corner}) 
of the lattice $L$ is a doubly-irreducible element 
in $L - \set{0,1}$ on the left (resp., right) boundary of~$L$. 
We define a \emph{rectangular lattice} $L$ as 
a planar semimodular lattice which has exactly one left weak corner, $u_l$,
and exactly one right weak corner, $u_r$,
and they are complementary, that is
\begin{align*}
   u_l \jj u_r &= 1,\\
   u_l \mm u_r &= 0.
\end{align*}

G. Cz\'edli and E.\,T. Schmidt \cite{CSb} define a \emph{patch lattice}
as a rectangular lattice in which the weak corners are dual atoms; in this note we allow $\SC 2$ as a patch lattice.

Let the lattice $L$ be a gluing of the lattices $A$ and $B$ over $I$
($I$ is a filter of~$A$ and an ideal of $B$). If $I$ is a chain, then
we say that $L$ is a gluing of the lattices $A$ and $B$ over a chain.

So here is the Cz\'edli-Schmidt \cite{CSb} result we reprove in this note:

\begin{theoremn}[A Structure Theorem of Planar Semimodular Lattices]
Let $L$ be a planar semimodular lattice with more than one element. 
Then there is a sequence
\[
L_1, L_2, \dots, L_n = L
\]
of lattices such that 
each~$L_i$, for $1 \leq i \leq n$, is a patch lattice
or $L_i$ is the gluing over a chain of  $L_j$ and $L_k$
for some $j,k < i$. 
\end{theoremn}

An equivalent form of the Theorem is the following statement:

\medskip

\emph{Let $L$ be a planar semimodular lattice with more than one element. 
If $L$ is not a patch lattice, 
then $L$ has an ideal $A$ and a filter $B$ 
\lp also planar semimodular lattices\rp
such that $C = A \ii B\neq \es$
is a chain and $L = A \uu B$ is the gluing of the lattices 
$A$ and $B$ over the chain~$C$.}

\medskip

We will refer to this statement also as the Theorem.

We make one more modification: 
we assume that $L$ be slim. (This concept was introduced in~\cite{GK07}; 
a~planar semimodular lattice is \emph{slim} if it contains no $\SM 3$, equivalently, no cover-preserving $\SM 3$.)
We can do this because if $L$ is not slim, then we can ``slim'' it 
by dropping the ``eyes'', 
the middle elements in cover-preserving \text{$\SM 3$-s.}
We~find the $A_s$ and $B_s$ for the slimmed $L$ 
and then put the eyes back into $A_s$ and $B_s$ to obtain
the $A$ and $B$ for $L$.

We call slim, planar, semimodular lattices SPS lattices.

\section{Reduction to slim rectangular lattices}
We prove now that in the Theorem we may assume that
$L$ is slim and rectangular.

We start with an easy result from the Appendix 
of G. Gr\"atzer and E. Knapp \cite[Lemma 7]{GK09}, 
which we state in a simplified form:

\begin{lemma}\label{L:onestepext}
Let $L$ be an SPS lattice with more than one element. 
Let $E = \set{a \prec b \prec c}$ be on the left boundary of $L$. 
Let $a$ be meet-irreducible and let $c$ be join-irreducible in $L$.
Order $L[E] = L \uu \set{t}$, $t \nin L$, 
so that $L$ be a suborder and $a \prec t \prec c$; 
place $t$ to the left of $E$, above $a$ and below~$c$
\lp for an illustration, see Figure~\ref{Fi:extension}\rp.

Then $L[E]$, called a \emph{one-step extension of $L$,} 
is an SPS lattice.
\end{lemma}

\begin{figure}[h!]
\centerline{\includegraphics{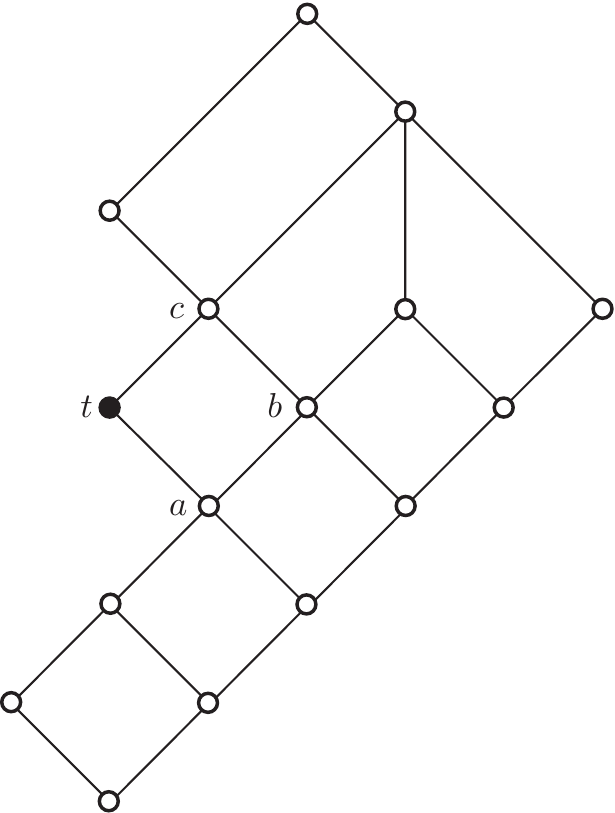}}
\caption{One-step extension}\label{Fi:extension}
\end{figure}

Now we can state the reduction step.

\begin{lemma}\label{L:L:onestepreduction}
Let $L$ be an SPS lattice  with more than one element
and let $L[E]$ be a one-step extension of~$L$, 
as in Lemma~\ref{L:onestepext}. 
If $L[E]$ can be obtained as a gluing of two lattices over a~chain,
then the same holds for $L$.
\end{lemma}

\begin{proof}
Let $L[E]$ have an ideal~$A$ and a filter $B$ 
such that $L[E] = A \uu B$ and
$C = A \ii B\neq \es$ is a chain. 
Set $A' = A \ii L = A - \set{t}$, 
$B' = B \ii L = B - \set{t}$, $C' = C \ii L = C - \set{t}$. 
Since $0 \in A'$ and $t$ is doubly irreducible, 
it follows that $A' \neq \es$ is an ideal of $L$; 
similarly, $B' \neq \es$ is a filter of $L$. 

If $C' = \es$, then $C = \set{t}$, 
implying that $A = {\Dg t}$ and $B = {\Ug t}$.
Since 
\[
   b \nin {\Dg t} \uu {\Ug t} = A \uu B = L[E],
\]
this is a contradiction, proving that $C' \neq \es$.

We conclude that $L = A' \uu B'$ is the gluing of the lattices 
$A'$ and $B'$ over the chain $C' = A' \ii B'$.
\end{proof}

By G. Gr\"atzer and E.~Knapp~\cite[Theorem 7]{GK09}, applying the one-step extension as many times as necessary
we obtain a rectangular lattice. So we get the following result:

\begin{lemma}[Reduction step]\label{T:rectangular}
Let $L$ be an SPS lattice with more than one element. 
Then there exists a slim rectangular extension 
$\Lh$ satisfying the following condition: 
if $\Lh$ can be obtained as a gluing of two lattices over a chain,
then the same holds for $L$. 
\end{lemma}

\section{Proving the Theorem for slim rectangular lattices}
For a rectangular lattice $L$, 
let $x$ be on the upper left boundary of $L$ with $x \nin \set{u_l, 1}$.
Define (see Figure~\ref{Fi:L}):
\begin{align*}
   L_{\tup{top}}    &= [x \mm u_r, 1],\\
   L_{\tup{bottom}} &= [0, x].
\end{align*}

\begin{figure}[t]
\centerline{\includegraphics{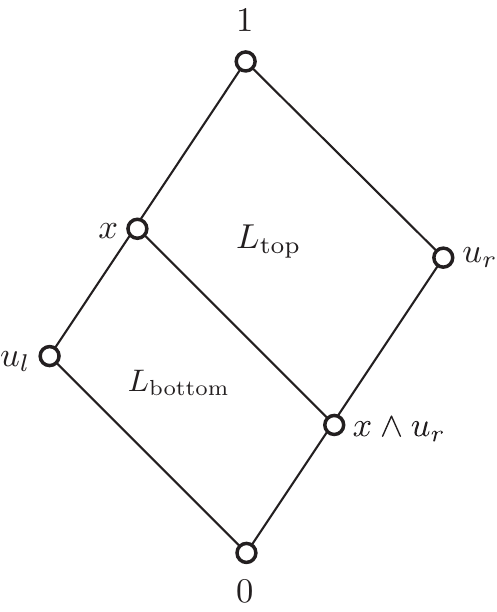}}
\caption{Decomposing a rectangular lattice $L$}\label{Fi:L}
\end{figure}

We now present a result of G. Gr\"atzer and E.~Knapp~\cite[Lemma 14 and Theorem 18]{GK10} in a simplified form:

\begin{theorem}[Decomposition Theorem]\label{T:decomp}
Let $L$ be a slim rectangular lattice
and let $x \nin \set{u_l, 1}$ be on the upper left boundary of $L$. 
From the lattice $L$, we construct the slim rectangular lattices 
$L_{\tup{top}}$ and $L_{\tup{bottom}}$. 
Then $L$ can be reconstructed 
from these two lattices by gluing them over a chain $[x \mm u_r, x]$.
In addition, $x = u_l \jj (x \mm u_r)$.
\end{theorem}

Now we are ready to prove the Theorem 
for a slim rectangular lattice $L$. 
We~proceed by induction on the size of $L$. 
If $L$ is a patch lattice, there is nothing to prove.
If $L$ is not a patch lattice, then by symmetry, 
we can assume that $u_l$ is not a dual atom. 
Take an element $x \in L$ with $u_l < x < 1$.

By Theorem~\ref{T:decomp}, we obtain $L$ by gluing 
the smaller rectangular lattices
$L_{\tup{top}}$ and $L_{\tup{bottom}}$ over a chain.
By the induction hypothesis, there is a sequence of lattices
$T_1, T_2, \dots, T_n = L_{\tup{top}}$ for $L_{\tup{top}}$
as required by the Theorem,
and another  $B_1, B_2, \dots, B_n = L_{\tup{bottom}}$ 
for $L_{\tup{bottom}}$.
Merging the two sequences and appending $L$ to the end,
we get the sequence for $L$, concluding the proof of the Theorem
for rectangular lattices.

\end{document}